\documentclass[11pt]{amsart}

\usepackage[utf8]{inputenc}
\usepackage[T1]{fontenc}
\usepackage{amsmath,amssymb,amsthm,mathtools}
\usepackage{enumitem}
\usepackage{booktabs}
\usepackage{array}
\usepackage{longtable}
\usepackage[unicode=true,
bookmarks=true,bookmarksnumbered=true,bookmarksopen=true,bookmarksopenlevel=2,
breaklinks=false,pdfborder={0 0 1},backref=false,colorlinks=false]
{hyperref}
\hypersetup{colorlinks,linkcolor=blue,anchorcolor=blue,citecolor=blue}
\usepackage{cite}

\numberwithin{equation}{section}
\newtheorem{theorem}{Theorem}[section]
\newtheorem{proposition}[theorem]{Proposition}
\newtheorem{lemma}[theorem]{Lemma}

\theoremstyle{definition}
\newtheorem{definition}[theorem]{Definition}
\newtheorem{externaltheorem}[theorem]{External theorem}
\theoremstyle{remark}

\newcommand{\Om}{\Omega_n}
\newcommand{\E}{\mathbb E}
\newcommand{\BMO}{\mathrm{BMO}}
\newcommand{\bmo}{\mathrm{bmo}}
\DeclareMathOperator{\id}{id}
\DeclareMathOperator{\Tr}{Tr}
\DeclareMathOperator{\VN}{VN}
\DeclareMathOperator{\Fix}{Fix}
\newcommand{\cN}{\mathcal N}
\newcommand{\cM}{\mathcal M}
\newcommand{\cF}{\mathcal F}
\newcommand{\elln}{\ell_2^n}
\newcommand{\one}{\mathbf 1}

\title{The Endpoint Fractional Riesz Estimate on the Hamming Cube}

\author[Zhendong  Xu]{Zhendong Xu}
\address{
	Department of Mathematical Sciences and the Research Institute of Mathematics, Seoul National University, Gwanak-ro 1, Gwanak-gu, Seoul 08826, Republic of Korea}
\email{zhendong\_xu\_97@snu.ac.kr}

\author[Hao Zhang]{Hao Zhang}
\address{
	Instituto de Ciencias Matem\'{a}ticas, Consejo Superior de Investigaciones Cient\'{i}ficas, C/ Nicol\'{a}s Cabrera 13-15, 28049, Madrid, Spain}
\email{hao.zhang@icmat.es}
\date{}


\begin{document}
	
	\begin{abstract}
		Let \(1<p<2\) and \(D_j\) be the discrete partial derivative on the Hamming cube $\Omega_n = \{ -1, 1\}^n$. Let \(\Delta=\sum_{j=1}^nD_j\) be the discrete Laplacian.  
		We prove the endpoint inequality
		\[
		\left\|\left(\sum_{j=1}^n|D_jf|^2\right)^{1/2}\right\|_{p}
		\lesssim_p\|\Delta^{1/p}f\|_{p}, \quad \forall f:\Omega_n \to \mathbb C.
		\]
		This result answers the conjecture proposed by Naor, Eskenazis and Ivanisvili (see \cite{BenEfraimLustPiquard,IvanisviliVolberg} or \cite[Remark 45]{EskenazisIvanisvili}). The proof relies heavily on the noncommutative semigroup BMO theory \cite{JungeMei}.
		
	\end{abstract}
	
	\maketitle
	
	\section{Introduction}\label{sec:introduction}
	
	Dimension-free estimate for discrete Riesz transforms on the Hamming cube was first shown by Lust-Piquard in the remarkable work \cite{LP1998} and later was found great applications in the theory of sharp metric $X_p$ inequalities and Banach space embedding theory \cite{NA2016}. Further devolopments in this direction can be found in \cite{LP1999,LP2004,BenEfraimLustPiquard,JMP2018,DomelevoIvanisviliPetermichlVolberg} etc.  
	
	Let $\Omega_n = \{ -1, 1\}^n$ be the Hamming cube with normalized counting measure. For a function $f:\Omega_n \to \mathbb C$, define the discrete partial derivative $D_j$ ($1\leq j\leq n$) and the discrete Laplacian $\Delta$ by
	\[
	D_jf(x)=\frac{f(x)-f(x^{(j)})}{2},
	\qquad
	\Delta f = \sum_{j=1}^n D_j f,
	\]
	where $x^{(j)}$ denotes the point in $\Omega_n$ obtained by flipping the $j$-th coordinate of $x$. Denote by $\nabla = (D_1, \ldots, D_n)$ the discrete gradient. Lust-Picard showed that for every $2 \leq p < \infty$, the following dimension-free estimate holds:
	\[
	C_p^{-1}\| \Delta^{1/2} f \|_p \leq \| \nabla f \|_{L_p(\Om;\ell_2^n)} \leq C_p \| \Delta^{1/2}f \|_p,
	\]
	while the upper bound fails for $1 < p < 2$ via Lamberton's counterexample, see \cite{LP1998}. Naor and Schechtman proved that an estimate
	\[
	\|\nabla f\|_{L_p(\Om;\ell_2^n)}\le C_\beta\|\Delta^\beta f\|_{p},
	\qquad 1<p<2,
	\]
	with \(C_\beta\) independent of the dimension implies that
	\(\beta\ge 1/p\); see \cite[Lemma~5.5]{BenEfraimLustPiquard}.  In the complementary direction, when $1<p<2$, Eskenazis and Ivanisvili showed that for every
	\(\varepsilon\in(0,1/2)\),
	\begin{equation}\label{var}
	\|\nabla f\|_{L_p(\Om;\ell_2^n)}
	\lesssim \frac{1}{\varepsilon}
	\|\Delta^{1/p+\varepsilon}f\|_{p};
	\end{equation}
	 see \cite[Proposition~43]{EskenazisIvanisvili}. In addition, they \cite{EskenazisIvanisvili} also obtained
	\[
	\|\nabla f\|_{L_p(\Om;\ell_2^n)}
	\lesssim \log(n+1)
	\|\Delta^{1/p}f\|_{p}.
	\]
	After the communication with Eskenazis, we learned that Naor has a different original proof for \eqref{var}, which has not been published.
	Recently, Ivanisvili and Volberg\cite{IvanisviliVolberg} found another different proof of \eqref{var}. 
	Naor, Eskenazis and Ivanisvili conjectured the dimension-free endpoint estimate in \cite[Remark~45]{EskenazisIvanisvili}.  
	
	
	\medskip
	
	Our main result answers this conjecture affirmatively.
	
	\begin{theorem}\label{thm:endpoint}
		For every \(1<p<2\) and every complex-valued function \(f\)
		on \(\Omega_n=\{-1,1\}^n\),
		\begin{equation*}
			\|\nabla f\|_{L_p(\Om;\ell_2^n)}
			\lesssim_p\|\Delta^{1/p}f\|_{p}.
		\end{equation*}
	\end{theorem}
	
	The key ingredient of the proof is the noncommutative semigroup BMO
	theory of Junge and Mei established in \cite{JungeMei}. This paper is organized as follows. In Section~\ref{sec:cube}, we present the necessary preliminaries on the Hamming cube, noncommutative $L_p$ spaces, quantum Markov semigroup theory and complex interpolation theory. Section~\ref{sec:semigroup} is devoted to theory of semigroup BMO spaces introduce by Junge and Mei in \cite{JungeMei}. In the last section, the concrete proof of our main theorem is given.
	
	After the submission of our paper to arXiv, the authors \cite{JLZZ26} told us that they also found a fundamentally different method to prove Theorem \ref{thm:endpoint}. Throughout the paper, we use the notation $A \lesssim B$ (or equivalently $ B \gtrsim A$) to indicate that $A \leq CB$ for some absolute positive constant $C$. In addition, the symbol $A \approx B$ signifies that both these inequalities and their converses hold.
	
	\bigskip
	
	\section{Preliminaries}\label{sec:cube}
	
	\subsection{Hamming cube}
	
	Fix an integer $n\ge1$, write $[n]=\{1,\ldots,n\}$, and set
	\[
	\Om=\{-1,1\}^n,
	\]
	equipped with the normalized counting measure $\mu_n$, so every point has
	mass $2^{-n}$.  Thus, for a scalar function $f:\Om\to\mathbb C$,
	\[
	\E f=\int_{\Om}f\,d\mu_n=2^{-n}\sum_{x\in\Om}f(x),
	\qquad
	\|f\|_{p}=
	\left( \E |f|^p \right)^{1/p}
	\]
	when $1\le p<\infty$, and
	$\|f\|_{\infty}=\max_{x\in\Om}|f(x)|$.
	
	If $H$ is a finite-dimensional complex Hilbert space, then for a vector-valued function $F : \Omega_n \to H$,
	\[
	\|F\|_{L_p(\Om;H)}=
	\left(\E\|F(x)\|_H^p\right)^{1/p},
	\qquad
	\|F\|_{L_\infty(\Om;H)}=\max_x\|F(x)\|_H.
	\]
	Throughout, scalar and Hilbert-space inner products are assumed to be linear in their
	first variable.

	For $x=(x_1,\ldots,x_n)\in\Om$, let
	\[
	x^{(j)}=(x_1,\ldots,x_{j-1},-x_j,x_{j+1},\ldots,x_n),
	\]
	and define
	\begin{equation}
		(\sigma_jf)(x)=f(x^{(j)}),
		\qquad
		D_j=\frac{I-\sigma_j}{2}.
		\label{eq:coordinate-difference}
	\end{equation}
	The map $x\mapsto x^{(j)}$ is a measure-preserving involution. Hence
	$\sigma_j$ is a self-adjoint unitary on $L_2(\Om)$. It follows
	directly from \eqref{eq:coordinate-difference} that
	\begin{equation*}
		D_j^*=D_j,
		\qquad
		D_j^2=D_j.
	\end{equation*}
	Note that all discrete partial derivatives $D_1,\ldots,D_n$ commute with each other. The positive discrete Laplacian is defned as
	\begin{equation*}
		\Delta=\sum_{j=1}^nD_j^*D_j
		=\sum_{j=1}^nD_j.
	\end{equation*}
	We set
	\begin{equation*}
		\nabla f=(D_1f,\ldots,D_nf),
		\qquad
		|\nabla f|(x)=\left(\sum_{j=1}^n|D_jf(x)|^2\right)^{1/2}.
	\end{equation*}
	
	We now introduce the Walsh basis and the fractional powers of the Laplacian.
	Let $w_\varnothing=1.$ For $\emptyset\neq A\subseteq[n]$, define the Walsh character
	\[
	w_A(x)=\prod_{j\in A}x_j, \ \forall x\in \Omega_n.
	\]
	It is easy to verify that the Walsh characters form an orthonormal basis of $L_2(\Om)$.  Every $f : \Omega_n \to \mathbb C$ has
	the unique expansion
	\begin{equation*}
		f=\sum_{A\subseteq[n]}\widehat f(A)w_A,
		\qquad
		\widehat f(A)=\langle f,w_A\rangle=\E[f w_A],
	\end{equation*}
	and Parseval's identity reads
	\begin{equation*}
		\|f\|_2^2=\sum_{A\subseteq[n]}|\widehat f(A)|^2.
	\end{equation*}
	
	Since changing the $j$-th sign multiplies $w_A$ by $-1$ precisely when
	$j\in A$, one has
	\begin{equation}
		\sigma_jw_A=(-1)^{\one_{\{j\in A\}}}w_A,
		\qquad
		D_jw_A=\one_{\{j\in A\}}w_A,
		\qquad
		\Delta w_A=|A|w_A.
		\label{eq:walsh-eigenvalues}
	\end{equation}
	Here, $|A|$ denotes the cardinality of $A$. In particular, the kernel of $\Delta$ consists exactly of the constants.
	Define
	\begin{equation}
		L_p^0(\Om;H)=\{F\in L_p(\Om;H):\E F=0\}.
		\label{eq:mean-zero-space}
	\end{equation}
	The projection onto this space is $ \Pi_0=I-\E$.
	
	We now fix the convention for all fractional and complex powers.  If
	$z\in\mathbb C$, then
	\begin{equation}
		\Delta^zf:=
		\sum_{\varnothing\ne A\subseteq[n]}
		e^{z\log|A|}\widehat f(A)w_A.
		\label{eq:fractional-powers}
	\end{equation}
	Here $\log|A|$ is the real logarithm of the positive integer $|A|$. There is no domain issue in \eqref{eq:fractional-powers} since for fixed $n$ the spectrum is the
	finite set $\{0,1,\ldots,n\}$. Thus every $\Delta^z$ annihilates constants; in
	particular $\Delta^0=\Pi_0$. On $L_p^0(\Om;H)$,
	$\Delta^0=I$, every $\Delta^z$ is invertible, and
	\begin{equation}
		\Delta^z\Delta^w=\Delta^{z+w},
		\qquad
		D_j\Delta^z=\Delta^zD_j.
		\label{eq:spectral-identities}
	\end{equation}
	In addition, $(\Delta^z)^*=\Delta^{\overline z} \ \text{on } L_2^0(\Om;H).$
	
	\subsection{Noncommutative $L_p$ spaces}
	We recall fundamental theory on noncommutative $L_p$ spaces on semifinite von Neumann algebras. Let $(\cN,\tau)$ be a semifinite von Neumann algebra with a normal semifinite faithful trace $\tau$. Let $ \mathcal S_+$ denote the set of all positive $x \in \cN$ such that $\tau({\rm supp}\, x) < \infty$, where ${\rm supp}\, x$ denotes the support of $x$. Let $\mathcal S$ be the linear span of $\mathcal S_+$. Then $\mathcal S$ is a $*$-subalgebra of $\cN$ which is weak-$*$ dense in $\cN$. Moreover, for $0 < p < \infty$ and $x \in \mathcal S$, we have $|x|^p \in \mathcal S_+$, where $|x| = (x^* x)^{ \frac{1}{2} }$, and thus $\tau (|x|^p) < \infty$. We define
	the associated
	noncommutative norms by
	\begin{equation*}
		\|x\|_{L_p(\cN)}=\tau(|x|^p)^{1/p}\quad(1\le p<\infty),
		\qquad
		\|x\|_{L_\infty(\cN)}=\|x\|_{\cN}.
	\end{equation*}
	The completion of $\mathcal S$ under $\|\cdot\|_{L_p(\cN)}$ is the noncommutative $L_p$ space associated with $(\cN, \tau)$, denoted by $L_p(\cN, \tau)$ or simply $L_p(\cN)$. We set $L_\infty(\cN) = \cN$ equipped with the operator norm. We refer the interested reader to \cite{PX2003,Xu2007book} for more information on noncommutative $L_p$ spaces.
	
	A von Neumann algebra $\cN$ is finite if the trace $\tau$ is finite for all positive elements. We normalize finite traces by $\tau(1)=1$ unless
	an unnormalized matrix trace is explicitly specified.
	
	
	Denote by
	\[
	\cN=L_\infty(\Om),
	\qquad \tau(f)=\E f.
	\]
	Let $M_m$ denote the algebra of $m\times m$ matrices. For matrix-valued functions we use
	\[
	\cN_m=M_m\,\overline\otimes\,L_\infty(\Om)
	\cong L^\infty(\Om;M_m)
	\]
	with trace $\tau_m=\Tr\otimes\E$, where $\Tr$ is the unnormalized matrix
	trace.
	
	\begin{definition}
		A linear map $S:\cN\to\cN$ is positive if $x\ge0$ implies $Sx\ge0$.
		It is completely positive if, for every $m\ge1$, the entrywise matrix
		amplification
		\[
		\id_{M_m}\otimes S:M_m\overline\otimes\cN
		\longrightarrow M_m\overline\otimes\cN
		\]
		is positive. It is $2$-positive if the amplification with $m=2$ is
		positive. It is unital if $S1=1$, and normal if it preserves
		suprema of increasing bounded nets of positive elements.
	\end{definition}
	 
	 It is clear that complete positivity implies $2$-positivity. We will use the following Kadison--Schwarz inequality.
	
	\begin{lemma}\label{ext:KS}
		If $S$ is a unital $2$-positive map between $\cN$, then
		\begin{equation}
			S(x)^*S(x)\le S(x^*x)
			\label{eq:kadison-schwarz}
		\end{equation}
		for every $x\in \cN$.  In particular, \eqref{eq:kadison-schwarz} holds for every unital completely
		positive map.
	\end{lemma}
	
	
	\subsection{Quantum Markov semigroups}
	We first introduce the notion of quantum Markov semigroups in \cite[Section~1.2]{JungeMei}.
	\begin{definition}
		Let $(\cN,\tau)$ be a semifinite von Neumann algebra.  A family
		$(T_t)_{t\ge0}$ of linear operators on $\cN$ is a quantum Markov semigroup if
		\begin{enumerate}[label=\textup{(\roman*)}]
			\item $T_0=I$ and $T_{s+t}=T_sT_t$;
			\item every $T_t$ is normal, completely positive, and unital;
			\item every $T_t$ is self-adjoint with respect to the trace:
			$\tau(T_t(x)y)=\tau(xT_t(y))$ whenever the expression is well-defined;
			\item for each $x\in\cN$, $T_t x\to x$ in the strong operator topology as $t\downarrow0$.
		\end{enumerate}
	\end{definition}

	Let
	\[
	{\rm dom(A)} = \left\{ x \in \cN : \lim_{t \to 0_+} \frac{x - T_t(x)}{t} \text{ converges in } \cN \right\}.
	\]
	Then for any $x \in {\rm dom}(A)$, the infinitesimal generator of $(T_t)$ is defined as
	\[
	A(x) = \lim_{t \to 0_+} \frac{x - T_t(x)}{t},
	\]
	where the limit is taken in the strong operator topology of $\cN$. 
	
	
	In particular, define the cube heat semigroup and its subordinated Poisson semigroup by
	\begin{equation*}
		T_t=e^{-t\Delta},
		\qquad
		P_s=e^{-s\sqrt\Delta}.
	\end{equation*}
	They act on the Walsh basis by
	\begin{equation}
		T_tw_A=e^{-t|A|}w_A,
		\qquad
		P_sw_A=e^{-s\sqrt{|A|}}w_A.
		\label{eq:semigroup-walsh}
	\end{equation}
	
	We verify that the cube heat semigroup is a quantum Markov semigroup.  Since the
	$D_j$'s commute,
	\[
	T_t=\prod_{j=1}^ne^{-tD_j}.
	\]
	The spectral projections of $\sigma_j$ are $(I+\sigma_j)/2$ and
	$(I-\sigma_j)/2=D_j$, so
	\begin{equation}
		e^{-tD_j}
		=\frac{I+\sigma_j}{2}+e^{-t}\frac{I-\sigma_j}{2}
		=\frac{1+e^{-t}}2I+\frac{1-e^{-t}}2\sigma_j.
		\label{eq:one-coordinate-heat}
	\end{equation}
	Both coefficients in \eqref{eq:one-coordinate-heat} are nonnegative and sum to one.  Each
	$\sigma_j$ is a trace-preserving $*$-automorphism, and a convex
	combination of $*$-automorphisms is normal, completely positive, unital,
	and trace preserving.  Expanding the product of \eqref{eq:one-coordinate-heat} expresses $T_t$ as
	a convex combination of the commuting $*$-automorphisms
	$\prod_{j\in B}\sigma_j$ $(B\subseteq [n])$; hence it has the same properties.  Each such
	automorphism is self-adjoint for the trace, hence so is $T_t$.  The
	semigroup law follows from the exponential formula, and the strong
	continuity is clear.  Thus the cube heat semigroup is a quantum Markov semigroup.
	
	For $s>0$, the subordination formula is
	\begin{equation*}
		P_s=\frac{s}{2\sqrt\pi}\int_0^\infty
		u^{-3/2}e^{-s^2/(4u)}T_u\,du.
	\end{equation*}
	This shows that every $P_s$ is also
	normal, completely positive, unital, trace preserving, and self-adjoint.  From \eqref{eq:semigroup-walsh}, on the mean-zero subspace,
	\begin{equation}
		\partial_sP_s=-\sqrt\Delta\,P_s.
		\label{eq:poisson-derivative}
	\end{equation}
	
	All preceding assertions remain valid on $\cN_m$ after applying the matrix
	amplification. Indeed, the amplifications of the
	$*$-automorphisms in \eqref{eq:one-coordinate-heat} are again $*$-automorphisms.
	
	\subsection{Markov dilation}
	
	Let $\cM$ be a semifinite von Neumann algebra. An increasing filtration is a family of von Neumann subalgebras
	$\cM_s\subseteq\cM_t$ for $s<t$ such that the trace restricted to each
	$\cM_s$ is semifinite and faithful and $\bigcup_{s\ge0}\cM_s$ is
	weak-operator dense in $\cM$.  A trace-preserving conditional
	expectation $E_s:\cM\to\cM_s$ is a normal positive unital idempotent that
	fixes $\cM_s$, preserves the trace, and satisifes:
	$$ E_s(axb)=aE_s(x)b,$$ 
	for $a,b\in\cM_s$ and $x \in \cM$.
	
	We introduce the notion of Markov dilation in \cite[Section~4]{JungeMei}.
	\begin{definition}
		A semigroup $(T_t)_{t\ge0}$ on $(\cN,\tau)$ admits a Markov
		dilation if there exist
		\begin{enumerate}[label=\textup{(\roman*)}]
			\item a larger semifinite von Neumann algebra $(\cM,\widetilde\tau)$;
			\item an increasing filtration $(\cM_s)_{s\ge0}$ of von Neumann
			subalgebras and the trace-preserving conditional expectation
			$E_s:\cM\to\cM_s$;
			\item trace-preserving $*$-homomorphisms
			$\pi_s:\cN\to\cM_s$;
		\end{enumerate}
		such that, for every $0\le s<t$ and $x\in\cN$,
		\begin{equation*}
			E_s(\pi_t(x))=\pi_s(T_{t-s}x).
		\end{equation*}
	\end{definition}
	
	
	\begin{lemma}\label{lem:dilation}
		The cube heat semigroup and every matrix amplification of it admit a Markov dilation.
	\end{lemma}
	
	\begin{proof}
	This can be immediately obtained from the fact that every quantum Markov semigroup on a finite von Neumann algebra admits a Markov dilation, see \cite{JRSh}.
	\end{proof}
	
	\subsection{Carr\'e du champ and the condition $\Gamma_2\ge0$}
	We now introduce the carr\'{e} du champ in the sense of Meyer and the condition $\Gamma_2\ge0$.
	\begin{definition}
		Let $A$ be the positive generator of a quantum Markov semigroup $(T_t)_{t\ge0}$.  On a common
		$*$-algebra contained in the domain of $A$, define the carr\'{e} du champ by
		\begin{equation}
			\Gamma_A(x,y)= \frac{1}{2} ( A(x^*)y+x^*A(y)-A(x^*y) ).
			\label{eq:gradient-form}
		\end{equation}
		The semigroup satisfies $\Gamma_2\ge0$ if
		\begin{equation}
			\Gamma_A(T_tx,T_tx)\le T_t\Gamma_A(x,x)
			\label{eq:gamma2}
		\end{equation}
		for every $t>0$ and every $x$ in that algebra.
	\end{definition}
	

	On $\Omega_n$, every function is in the generator domain.  By using
	$\sigma_i(x^*y)=\sigma_i(x)^*\sigma_i(y)$, we calculate
	\begin{align*}
		& \ \ \ \ D_i(x^*)y+x^*D_i(y)-D_i(x^*y)\\
		&=\frac12\bigl[(x^*-\sigma_i(x)^*)y
		+x^*(y-\sigma_i(y))-x^*y+\sigma_i(x)^*\sigma_i(y)\bigr]\\
		&=\frac12\bigl(x^*-\sigma_i(x)^*\bigr)
		\bigl(y-\sigma_i(y)\bigr)\\
		&=2(D_i x)^*D_i y.
	\end{align*}
	Summing over $i$ and comparing with \eqref{eq:gradient-form} yields
	\begin{equation}
		\Gamma_\Delta(x,y)=\sum_{i=1}^n(D_i x)^*D_i y.
		\label{eq:cube-gamma}
	\end{equation}
	This computation remains valid for matrix-valued funcrions $x,y$.
	
	We can now prove the $\Gamma_2 \geq 0$ condition directly for any matrix
	amplification of the cube heat semigroup. Since $T_t$ and all its matrix
	amplifications are unital and completely positive,
	commutation of $T_t$ with $D_i$, followed by Lemma~\ref{ext:KS}, gives
	\begin{equation*}
		\begin{aligned}
			\Gamma_\Delta(T_tx,T_tx)
			&=\sum_i(T_tD_i x)^*(T_tD_i x)\\
			&\le\sum_iT_t\bigl((D_i x)^*D_i x\bigr)\\
			&=T_t\Gamma_\Delta(x,x).
		\end{aligned}
	\end{equation*}
	Thus \eqref{eq:gamma2} holds
	uniformly for scalar and matrix-valued functions.

	\subsection{Complex interpolation}
	We recall some basic theorems about complex interpolation of Banach spaces. 
	\begin{definition}
		A compatible Banach couple $(X_0,X_1)$ consists of two Banach spaces
		continuously embedded in a common Hausdorff topological vector space.
		Let $\mathbb S=\{z\in\mathbb C:0\le\Re z\le1\}$.  The Calder\'on class
		$\cF(X_0,X_1)$ consists of bounded continuous functions
		$F:\mathbb S\to X_0+X_1$ that are analytic in the open strip, satisfy
		$F(it)\in X_0$, $F(1+it)\in X_1$, and
		\[
		\|F\|_{\cF}
		=\max\left\{\sup_t\|F(it)\|_{X_0},\
		\sup_t\|F(1+it)\|_{X_1}\right\}<\infty.
		\]
		In the standard Calder\'on class one also requires $\|F(j+it)\|_{X_j}\to0$ as $|t|\to\infty$ for $j=0,1$. 
		For $0<\theta<1$,
		\[
		[X_0,X_1]_\theta
		=\{F(\theta):F\in\cF(X_0,X_1)\},
		\]
		with norm
		$\|x\|_\theta=\inf\{\|F\|_{\cF}:F(\theta)=x\}$.
	\end{definition}
	
	This is the Calder\'on complex method; see Bergh--L\"ofstr\"om
	\cite[Section 4.1]{BerghLofstrom}.  We state all consequences which will be needed later.
	
	\begin{theorem}\label{ext:basicinterpolation}
		Let $0<\theta<1$.
		\begin{enumerate}[label=\textup{(\roman*)}]
			\item If $S:X_j\to Y_j$ has norm at most $M_j$ for $j=0,1$, then
			\[
			S:[X_0,X_1]_\theta\to[Y_0,Y_1]_\theta,
			\qquad \|S\|\le M_0^{1-\theta}M_1^\theta.
			\]
			\item $[X_0,X_1]_\theta=[X_1,X_0]_{1-\theta}$ isometrically.
			\item If $\mathsf P$ is a bounded projection on $X_0$ and $X_1$, then
			\[
			[\mathsf PX_0,\mathsf PX_1]_\theta
			=\mathsf P[X_0,X_1]_\theta
			\]
			with equivalence constants controlled by its endpoint norms.
			\item If $(\Xi,\nu)$ is a finite measure space, $H$ is finite
			dimensional, $1\le p_0,p_1\le\infty$, and
			\[
			\frac1p=\frac{1-\theta}{p_0}+\frac\theta{p_1},
			\]
			then
			\[
			[L_{p_0}(\Xi;H),L_{p_1}(\Xi;H)]_\theta=L_p(\Xi;H)
			\]
			with the canonical interpolation norm.
		\end{enumerate}
	\end{theorem}
	
	Parts (i)--(iii) are the exactness, symmetry, and retract properties in
	Bergh--L\"ofstr\"om \cite[Chapter~4]{BerghLofstrom}; part (iv) is the $L_p$ interpolation theorem in \cite[Chapter~5]{BerghLofstrom}.
	
	\medskip
	
	The following is Stein's interpolation theorem \cite{SteinInterpolation}.
	
	\begin{theorem}\label{ext:Stein}
		Let $(X_0,X_1)$ and $(Y_0,Y_1)$ be compatible Banach couples.  Suppose
		$S_z:X_0+X_1\to Y_0+Y_1$, $0\le\Re z\le1$, is a family of linear maps
		that is norm-continuous on the closed strip, analytic in the open strip,
		and bounded there in the sum-space operator norm:
		\[
		\sup_{0\le\Re z\le1}
		\|S_z\|_{X_0+X_1\to Y_0+Y_1}<\infty.
		\]
		Assume also that
		\[
		\sup_{t\in\mathbb R}\|S_{it}\|_{X_0\to Y_0}\le M_0,
		\qquad
		\sup_{t\in\mathbb R}\|S_{1+it}\|_{X_1\to Y_1}\le M_1.
		\]
		Then
		\begin{equation*}
			S_\theta:[X_0,X_1]_\theta\longrightarrow[Y_0,Y_1]_\theta,
			\qquad
			\|S_\theta\|\le M_0^{1-\theta}M_1^\theta.
		\end{equation*}
	\end{theorem}

	\bigskip
	
	\section{Semigroup BMO and interpolation on the Hamming cube}\label{sec:semigroup}
	
	This section is devoted to operator-algebraic notion needed to invoke the noncommutative semigroup BMO theory established in \cite{JungeMei}. The main goal is to use such noncommutative semigroup theory to show Lemma \ref{lem:BMOtools}.
	
	\subsection{Semigroup BMO seminorms}\label{sec:BMO}
	
	Let $(T_t)_{t\geq 0}$ be a quantum Markov semigroup on $(\cN,\tau)$ with the positive generator $A$, let $P_t=e^{-t\sqrt A}$, and let $x\in\cN\cup L_2(\cN)$.  The column BMO seminorms associated with a semigroup $S=(S_t)_{t\geq 0}$ are
	\begin{align}
		& \|x\|_{\bmo_c(S)}
		=\sup_{t>0}\|S_t(x^*x)-S_t(x)^*S_t(x)\|_\infty^{1/2},
		\label{eq:bmo-small}\\
		& \|x\|_{\BMO_c(S)}
		=\sup_{t>0}\|S_t(|x-S_tx|^2)\|_\infty^{1/2}.
		\label{eq:bmo-big}
	\end{align}
	We use \eqref{eq:bmo-small}--\eqref{eq:bmo-big} with $S=P$ and mention the versions with $S=T$
	when stating the interpolation theorem.
	
	The differential column seminorms are
	\begin{align}
		\|x\|_{\BMO_c(\partial)}^2
		&=\sup_{t>0}\left\|P_t\int_0^t
		|\partial_sP_sx|^2s\,ds\right\|_\infty,
		\label{eq:bmo-partial}\\
		\|x\|_{\BMO_c(\Gamma)}^2
		&=\sup_{t>0}\left\|P_t\int_0^t
		\Gamma_A(P_sx,P_sx)s\,ds\right\|_\infty.
		\label{eq:bmo-gamma}
	\end{align}
	Define the space-time carr\'e du champ by
	\begin{equation}
		\widehat\Gamma(P_sx,P_sx)
		=\Gamma_A(P_sx,P_sx)+|\partial_sP_sx|^2
		\label{eq:space-time-gamma}
	\end{equation}
	and
	\begin{equation}
		\|x\|_{\BMO_c(\widehat\Gamma)}^2
		=\sup_{t>0}\left\|P_t\int_0^t
		\widehat\Gamma(P_sx,P_sx)s\,ds\right\|_\infty.
		\label{eq:bmo-hatgamma}
	\end{equation}
	
	For any column seminorm $X_c$, define the row seminorm and full seminorm
	by
	\begin{equation*}
		\|x\|_{X_r}=\|x^*\|_{X_c},
		\qquad
		\|x\|_X=\max\{\|x\|_{X_c},\|x\|_{X_r}\}.
	\end{equation*}
	Thus $\BMO(P)$, $\BMO(\partial)$, and
	$\BMO(\widehat\Gamma)$ below always mean full row and column BMO.
	
	The fixed-point algebra is
	\[
	\Fix(T)=\{x:T_tx=x\text{ for every }t\ge0\}.
	\]
	For the cube heat semigroup, \eqref{eq:semigroup-walsh} gives
	\[
	\lim_{t\to\infty}T_tf=\widehat f(\varnothing)w_\varnothing=\E f,
	\]
	so $\Fix(T)$ consists exactly of the constants and its complement is
	$L_p^0(\Om)$ from \eqref{eq:mean-zero-space}.  All seminorms \eqref{eq:bmo-small}--\eqref{eq:bmo-hatgamma} vanish on constants.
	Conversely, on the cube the differential seminorms vanish only on
	constants.  Indeed, if \eqref{eq:bmo-partial} is zero, then for every $t>0$ the positive
	function
	\[
	P_t\int_0^t|\partial_sP_sf|^2s\,ds
	\]
	is zero.  Note that $P_t$ on the finite function space is invertible since its eigenvalues in \eqref{eq:semigroup-walsh} are strictly positive.  Hence the
	integral is zero.  Its continuous nonnegative integrand is zero for
	$0<s<t$, so $\partial_sP_sf=0$; \eqref{eq:poisson-derivative} and \eqref{eq:semigroup-walsh} imply that $f$ is constant. The same argument, using \eqref{eq:cube-gamma}, shows
	that \eqref{eq:bmo-gamma} vanishes only on constants.  Thus these seminorms are norms on
	their unique mean-zero representatives.
	
	\subsection{Hilbert-valued versions and their matrix realization}
	
	For a finite-dimensional Hilbert space $H$ and $F:\Om\to H$, define
	\begin{align}
		\|F\|_{\BMO_{\partial,H}}^2
		&=\sup_{t>0}\left\|P_t\int_0^t
		\|\partial_sP_sF\|_H^2s\,ds\right\|_\infty,
		\label{eq:hilbert-bmo-partial}\\
		\|F\|_{\BMO_{\Gamma,H}}^2
		&=\sup_{t>0}\left\|P_t\int_0^t
		\sum_{i=1}^n\|D_iP_sF\|_H^2s\,ds\right\|_\infty.
		\label{eq:hilbert-bmo-gamma}
	\end{align}
	
	Write $H=\ell_2^m$, $F=(F_1,\ldots,F_m):\Om\to H$ and embed it as the first matrix column via the following map:
	\begin{equation}
		JF=\sum_{k=1}^me_{k1}\otimes F_k\in\cN_m.
		\label{eq:first-column-embedding}
	\end{equation}
	
		\begin{lemma}
		Let $H=\ell_2^m$ and $F=(F_1,\ldots,F_m):\Om\to H$. Then  for $1\le p\leq \infty$,
		$$ \|JF\|_{L_p(\cN_m)}
		=\|F\|_{L_p(\Om;H)}. $$
		\end{lemma}
		\begin{proof}
	By \eqref{eq:first-column-embedding}, one has
	\[
	(JF)^*(JF)=e_{11}\otimes\sum_{k=1}^m|F_k|^2.
	\]
	This gives
	\[
	((JF)^*(JF))^{p/2}
	=e_{11}\otimes\left(\sum_{k=1}^m|F_k|^2\right)^{p/2}.
	\]
	Therefore, for $1\le p<\infty$,
	\begin{align*}
		\|JF\|_{L_p(\cN_m)}^p
		&=\E\Tr\bigl(((JF)^*(JF))^{p/2}\bigr)\\
		&=\E\left(\sum_{k=1}^m|F_k|^2\right)^{p/2}
		=\|F\|_{L_p(\Om;H)}^p,
	\end{align*}
	while for $p=\infty$, this equality follows from the operator norm of a column.
	Thus $J$ is an isometry for every $p$.
	\end{proof}
	
	Because $P_s$, $\partial_sP_s$, and $D_i$ act on the cube variable only, one has
	\[
	|\partial_sP_s(JF)|^2
	=e_{11}\otimes\|\partial_sP_sF\|_H^2,
	\]
	and \eqref{eq:cube-gamma} implies
	\[
	\Gamma_\Delta(P_sJF,P_sJF)
	=e_{11}\otimes\sum_i\|D_iP_sF\|_H^2.
	\]
	A direct calculation after taking operator norms shows that 
	\[
	\| JF \|_{BMO_c(\partial)} = \| F \|_{BMO_{\partial, H}}, \qquad \| JF \|_{BMO_c(\Gamma)} = \| F \|_{BMO_{\Gamma, H}}.
	\]
	The following lemma is also needed.
	
	\begin{lemma}
		Let $H=\ell_2^m$ and $F=(F_1,\ldots,F_m):\Om\to H$. Then
		$$ \| JF \|_{BMO(\partial)}=\| JF \|_{BMO_c(\partial)} , \qquad \| JF \|_{BMO(\Gamma)}=\| JF \|_{BMO_c(\Gamma)}.   $$
		In particular, $J$ is isometric from $\BMO_{\partial, H}$ to $\BMO(\partial)$.
		\end{lemma}
		\begin{proof}
			We only show $ \| JF \|_{BMO(\partial)}=\| JF \|_{BMO_c(\partial)}$ as the other equality can be obtained in the same way. It suffices to show
			$$ \| (JF)^* \|_{BMO_c(\partial)} \leq \| JF \|_{BMO_c(\partial)}= \| F \|_{BMO_{\partial, H}}. $$
			We need to use the following elementary inequality:
			$$ |A|\leq \|A\|_{M_m}\le \Tr(|A|), \quad \forall A\in M_m.  $$
			Therefore, one has
				\begin{align*}
					\| (JF)^* \|_{BMO_c(\partial)}^2&=\sup_{t>0}\left\|P_t\int_0^t
					|\partial_sP_s (JF)^*|^2s\,ds\right\|_\infty\\
					&\leq \sup_{t>0}\left\|P_t\int_0^t
					\text{Tr}\left(|\partial_sP_s (JF)^*|^2\right)s\,ds\right\|_\infty\\
					&=\sup_{t>0}\left\|P_t\int_0^t
					\|\partial_sP_sF\|_H^2s\,ds\right\|_\infty\\
					&=\|F\|_{\BMO_{\partial,H}}^2,
					\end{align*}
					which implies the desired result.
			\end{proof}
	
	By \cite[Theorem~2.9]{JungeMei} and \cite[Lemma~3.2]{JungeMei}, we have the following equivalence of BMO norms and multiplier theory under the $\Gamma_2\ge0$ condition.
	\begin{lemma}\label{ext:comparison}
		Let $(T_t)_{t\geq 0}$ be a quantum markov semigroup on $(\cN,\tau)$ and $(P_t)_{t\geq 0}$ be its subordinated
		Poisson semigroup, and assume $\Gamma_2\ge0$. Then, on $\cN\cup L^2(\cN)$, the three column seminorms
		\[
		\BMO_c(P),\qquad \bmo_c(P),\qquad
		\BMO_c(\widehat\Gamma)
		\]
		are equivalent with absolute constants.
	\end{lemma}
	
	\begin{lemma}\label{ext:multiplier}
		Let $(T_t)_{t\geq 0}$ be a quatum Markov semigroup on $(\cN,\tau)$ and $(P_t)_{t\geq 0}$ be its subordinated
		Poisson semigroup, and assume $\Gamma_2\ge0$. Let $a:(0,\infty)\to\mathbb C$ be measurable.
		Assume
		\begin{equation}
			c_a^2=\sup_{s>0}s\int_s^\infty
			\frac{|a(v-s)|^2}{v^2}\,dv<\infty.
			\label{eq:laplace-symbol-condition}
		\end{equation}
		Define
		\begin{equation*}
			M_af=\int_0^\infty a(t)\,\partial_tP_tf\,dt.
		\end{equation*}
		Then
		\begin{equation*}
			\|M_af\|_{\BMO_c(\Gamma)}
			\lesssim c_a\|f\|_{\BMO_c(\Gamma)}.
		\end{equation*}
	\end{lemma}
	
	Lemam \ref{ext:comparison} and Lemma \ref{ext:multiplier} are vital in the proof of Lemma \ref{lem:BMOtools}.
	
	
	\subsection{Interpolation of semigroup BMO spaces}
	
	For a quantutm Markov semigroup $(T_t)_{t\geq 0}$ on $(\cN, \tau)$ with the fixed-point projection
	$E_{\Fix}=\lim_{t\to\infty}T_t$ on $L_p(\cN)$ $(1<p<\infty)$, define
	\[
	L_p^0(\cN)=\ker E_{\Fix}.
	\]
	Equivalently this is the complemented part where $T_tx\to0$ for $x\in L_p(\cN)$. From \cite[Theorem 5.2(i)]{JungeMei}, we have the following interpolation theorem for noncommutative semigroup BMO spaces, which will also be utilized in the proof of Lemma \ref{lem:BMOtools}.
	\begin{theorem}
		\label{ext:JMinterpolation}
		Let $(T_t)_{t\geq 0}$ be a quantum Markov semigroup on $(\cN, \tau)$ and
		assume it admits a Markov dilation.  Let $p\ge1$, $q>1$,
		and
		\[
		X\in\{\BMO(T),\BMO(P),\BMO(\widehat\Gamma),
		\BMO(\partial),\bmo(P)\}.
		\]
		Then
		\begin{equation*}
			L_{pq}^0(\cN)=[X,L_p^0(\cN)]_{1/q}
		\end{equation*}
		with equivalence constants depending only on $p, q$, and not on the von Neumann algebra or semigroup.
	\end{theorem}
	
	For the cube heat semigroup, $E_{\Fix}=\E$, so this new notation agrees with the original defition \eqref{eq:mean-zero-space}.

	For $\Re z>0$, define the gamma function by
	\[
	\Gamma(z)=\int_0^\infty r^{z-1}e^{-r}\,dr.
	\]
	It extends meromorphically and obeys
	\[
	\Gamma(1+z)=z\Gamma(z),
	\qquad
	\Gamma(z)\Gamma(1-z)=\frac{\pi}{\sin(\pi z)}
	\]
	whenever both sides are defined, see \cite[Sections 5.2 and 5.5]{DLMF}.
	
	\medskip
	
	We end this section with the following lemma which is essential to the proof of Theorem \ref{thm:endpoint}.
	\begin{lemma}\label{lem:BMOtools}
		For every finite-dimensional Hilbert space $H$, the following hold uniformly in $n$ and $\dim H$.
		\begin{enumerate}[label=\textup{(\roman*)}]
			\item For any $F\in L_\infty(\Om;H)$,
			\begin{equation}
				\|F\|_{\BMO_{\Gamma,H}}
				\lesssim\|F\|_{L_\infty(\Om;H)}.
				\label{eq:linfty-bmo-gamma}
			\end{equation}
			\item For $u\in\mathbb R$ and mean-zero $F$,
			\begin{equation}
				\|\Delta^{-iu}F\|_{\BMO_{\Gamma,H}}
				\lesssim e^{\pi|u|}\|F\|_{\BMO_{\Gamma,H}}.
				\label{eq:imaginary-power-bmo}
			\end{equation}
			\item If $2<q<\infty$, then
			\begin{equation}
				[L_2^0(\Om;H),\BMO_{\partial,H}]_{1-2/q}
				=L_q^0(\Om;H)
				\label{eq:hilbert-interpolation}
			\end{equation}
			with equivalence constants depending only on $q$.
		\end{enumerate}
	\end{lemma}
	
	\begin{proof}
		Without loss of generality, assume $H=\ell_2^m$.
		For (i), by \eqref{eq:bmo-hatgamma} and Lemam \ref{ext:comparison}, one has
		\[
		\| F \|_{BMO_{\Gamma, H}} = \|JF\|_{\BMO_c(\Gamma)}
		\le\| JF\|_{\BMO_c(\widehat\Gamma)}\approx \|JF\|_{\BMO_c(P)}.
		\]
		Recall that \eqref{eq:bmo-big} gives
		\[
		\|JF\|_{\BMO_c(P)}
		=\sup_{t>0}\|P_t|JF-P_tJF|^2\|_{L_\infty(\cN_m)}^{1/2}.
		\]
		Applying the triangle inequality gives
		$$\|JF-P_tJF\|_{L_\infty(\cN_m)}\le2\| JF\|_{L_\infty(\cN_m)}.$$ 
		Since $P_t$ is contractive on $L_\infty(\cN_m)$, we obtain
		\[
		\|P_t|JF-P_tJF|^2\|_{L_\infty(\cN_m)}\le4\|JF\|_{L_\infty(\cN_m)}^2.
		\]
		So 
		$$ \|JF\|_{\BMO_c(P)}\le 2\|JF\|_{L_\infty(\cN_m)} = 2 \| F \|_{L_\infty(\Om;H)}. $$
		This proves \eqref{eq:linfty-bmo-gamma}.
		
		For (ii), define $a_u(v)=v^{2iu}=e^{2iu\log v}$ for $v>0$. Since $|a_u(v)|=1$,
		\[
		s\int_s^\infty\frac{|a_u(v-s)|^2}{v^2}\,dv
		=s\int_s^\infty v^{-2}\,dv=1.
		\]
		Thus \eqref{eq:laplace-symbol-condition} holds with $c_{a_u}=1$. For $A\subseteq [n]$ with $\lambda=|A|>0$, by using \eqref{eq:poisson-derivative}, one has
		\begin{align*}
			M_{a_u}w_A
			&=-\sqrt\lambda\int_0^\infty
			t^{2iu}e^{-t\sqrt\lambda}\,dt\,w_A\\
			&=-\lambda^{-iu}\int_0^\infty r^{2iu}e^{-r}\,dr\,w_A\\
			&=-\Gamma(1+2iu)\lambda^{-iu}w_A.
		\end{align*}
		On constants both sides vanish, and so
		\begin{equation*}
			M_{a_u}=-\Gamma(1+2iu)\Delta^{-iu}.
		\end{equation*}
		Lemma~\ref{ext:multiplier} now gives
		\begin{equation}
			\|\Delta^{-iu}F\|_{\BMO_{\Gamma,H}}
			\lesssim\frac{1}{|\Gamma(1+2iu)|}
			\|F\|_{\BMO_{\Gamma,H}}.
			\label{eq:gamma-multiplier-bound}
		\end{equation}
		
		For $y\ne0$, the recurrence and reflection identities of the Gamma function, together with
		$\sin(i\pi y)=i\sinh(\pi y)$, give
		\begin{align*}
			\Gamma(1+iy)\Gamma(1-iy)
			&=iy\,\Gamma(iy)\Gamma(1-iy)\\
			&=iy\,\frac{\pi}{\sin(i\pi y)}
			=\frac{\pi y}{\sinh(\pi y)}.
		\end{align*}
		The Euler integral also gives
		$\Gamma(1-iy)=\overline{\Gamma(1+iy)}$. Hence, for real $y$,
		\begin{equation*}
			|\Gamma(1+iy)|^2
			=\frac{\pi|y|}{\sinh(\pi|y|)},
		\end{equation*}
		with value $1$ at $y=0$ by continuity.  On $|y|\le2$ its reciprocal is
		bounded.  On $|y|>2$, using
		$\sinh(\pi|y|)\le e^{\pi|y|}/2$,
		\[
		\frac1{|\Gamma(1+iy)|}
		\lesssim\frac{e^{\pi|y|/2}}{\sqrt{|y|}}
		\lesssim e^{\pi|y|/2}.
		\]
		Set $y=2u$ in \eqref{eq:gamma-multiplier-bound} to obtain \eqref{eq:imaginary-power-bmo}.
		
		For (iii), apply Theorem~\ref{ext:JMinterpolation} to the finite
		algebra $\cN_m$. Lemma~\ref{lem:dilation} implies that any matrix amplification of the cube heat semigroup admits a Markov dilation.  Choose
		$r=2$, $s=q/2>1$, and $X=\BMO(\partial)$.  Then $rs=q$ and $1/s=2/q$,
		so
		\[
		L_q^0(\cN_m)
		=[\BMO(\partial),L_2^0(\cN_m)]_{2/q}.
		\]
		Theorem~\ref{ext:basicinterpolation}(ii) reverses the couple
		\begin{equation}
			L_q^0(\cN_m)
			=[L_2^0(\cN_m),\BMO(\partial)]_{1-2/q}.
			\label{eq:matrix-interpolation-reversed}
		\end{equation}
		Define the first-column projection $\mathsf P: L_\infty(\cN_m) \longrightarrow L_\infty(\cN_m) $ by
		$$ \mathsf{P}\left(\sum_{1\leq i, j\leq m} e_{ij}\otimes f_{ij}\right)=\sum_{i=1}^m e_{i1}\otimes f_{i1}.  $$
		It is clear that $\mathsf{P}$ is a contraction both on $L_2^0(\cN_m)$ and on $\BMO(\partial)$. Then by Theorem~\ref{ext:basicinterpolation}(iii) and \eqref{eq:matrix-interpolation-reversed},
		$$ 	\mathsf{P}L_q^0(\cN_m)
		=[\mathsf{P}L_2^0(\cN_m),\mathsf{P}\BMO(\partial)]_{1-2/q}.    $$
		Note that the embedding $J: L_\infty^0(\Omega_n; H) \longrightarrow \cN_m$ defined in \eqref{eq:first-column-embedding} is isometric both from $ L_p^0(\Omega_n; H)$ to $L_p^0(\cN_m)$ and from $\BMO_{\partial, H}$ to $\BMO(\partial)$. Hence,
		$$  (J^{-1}\mathsf{P})L_q^0(\cN_m)
		=[(J^{-1}\mathsf{P})L_2^0(\cN_m),(J^{-1}\mathsf{P})\BMO(\partial)]_{1-2/q},  $$
		which is exactly \eqref{eq:hilbert-interpolation}.
	\end{proof}
	
	\bigskip
	
	\section{Proof of Theorem~\ref{thm:endpoint}}\label{sec:diagonal}
	We prove the main theorem in this section.  The first step is a diagonal estimate for the Riesz transforms.
	\begin{lemma}\label{lem:diagonal}
		Let $2<q<\infty$ and set
		\begin{equation*}
			a=\frac12-\frac1q.
			\label{eq:critical-exponent-a}
		\end{equation*}
		For every $G=(g_1,\ldots,g_n):\Om\to\elln$,
		\begin{equation}
			\left\|\bigl(D_j\Delta^{-a}g_j\bigr)_{j=1}^n
			\right\|_{L_q(\Om;\elln)}
			\lesssim_q\|G\|_{L_q(\Om;\elln)}.
			\label{eq:diagonal-estimate}
		\end{equation}
	\end{lemma}
	
	\begin{proof}
		First suppose every $g_j$ has mean zero, and denote $H=\elln$.  On the
		strip $0\le\Re z\le1/2$, define
		\begin{equation*}
			U_zG=\bigl(D_j\Delta^{-z}g_j\bigr)_{j=1}^n.
		\end{equation*}
		By \eqref{eq:fractional-powers}, it is clear that $z\mapsto U_z$ is
		entire.
		
		Let $t\in\mathbb R$. By using \eqref{eq:walsh-eigenvalues} and
		$|\,|A|^{-it}|=1$, one has
		\begin{equation*}
			\begin{aligned}
				\|U_{it}G\|_2^2
				&=\sum_{j=1}^n\|D_j\Delta^{-it}g_j\|_2^2 =\sum_{j=1}^n
				\sum_{\substack{A\subseteq[n]\\j\in A}}
				\left||A|^{-it}\widehat g_j(A)\right|^2\\
				&=\sum_{j=1}^n
				\sum_{\substack{A\subseteq[n]\\j\in A}}
				|\widehat g_j(A)|^2\le\sum_{j=1}^n\sum_{A\subseteq[n]}|\widehat g_j(A)|^2
				=\|G\|_2^2.
			\end{aligned}
		\end{equation*}
		Thus $U_{it}$ is a contraction on $L_2^0(\Om;H)$.

		For an $H$-valued function $F$, define
		\[
		RF=\bigl(D_i\Delta^{-1/2}F\bigr)_{i=1}^n,
		\]
		with values in $\ell_2^n(H)$.  By \eqref{eq:spectral-identities} and
		\eqref{eq:poisson-derivative}, one has for each $i$,
		\begin{equation*}
			\partial_sP_s(D_i\Delta^{-1/2}F)
			=-\sqrt\Delta P_sD_i\Delta^{-1/2}F
			=-D_iP_sF.
		\end{equation*}
		Therefore, for every $t>0$,
		\begin{align*}
			P_t\int_0^t
			\|\partial_sP_s(RF)\|_{\ell_2^n(H)}^2s\,ds = P_t\int_0^t
			\sum_{i=1}^n\|D_iP_sF\|_H^2s\,ds.
		\end{align*}
		By \eqref{eq:hilbert-bmo-partial}--\eqref{eq:hilbert-bmo-gamma}, one has
		\begin{equation}
			\|RF\|_{\BMO_{\partial,\ell_2^n(H)}}
			=\|F\|_{\BMO_{\Gamma,H}}.
			\label{eq:riesz-bmo-identity}
		\end{equation}
		
		Now identify an element of $\ell_2^n(H)=\ell_2^n(\ell_2^n)$ with an array
		$(h_{ij})_{1\le i,j\le n}$ and define
		\[
		Q(h_{ij})=(h_{jj})_{j=1}^n\in H.
		\]
		Then
		\[
		\|Q(h_{ij})\|_H^2=\sum_j|h_{jj}|^2
		\le\sum_{i,j}|h_{ij}|^2,
		\]
		so $\|Q\|_{\ell_2^n(H) \rightarrow H}\le1$.  In the following, we consider the amplification $Q\otimes {\rm id}_{L_p(\Omega_n)}$. For a slight abuse of notation, we still denote by $Q$ this amplification as it will not cause any confusion. Since $Q$ acts only on Hilbert coordinates, it commutes
		with $P_s$ and $\partial_sP_s$.  Directly from \eqref{eq:hilbert-bmo-partial}, for any $K\in \BMO_{\partial,\ell_2^n(H)}$, one has
		\begin{equation}
			\|QK\|_{\BMO_{\partial,H}}
			\le\|K\|_{\BMO_{\partial,\ell_2^n(H)}}.
			\label{eq:diagonal-contraction}
		\end{equation}
		For real $t$, the $j$th coordinate of $QR(\Delta^{-it}G)$ is
		$D_j\Delta^{-1/2}\Delta^{-it}g_j$.  Therefore
		\begin{equation}
			U_{1/2+it}G=QR(\Delta^{-it}G).
			\label{eq:upper-factorization}
		\end{equation}
		Equations \eqref{eq:riesz-bmo-identity}--\eqref{eq:upper-factorization} and
		Lemma~\ref{lem:BMOtools}(i)--(ii) imply
		\begin{equation*}
			\begin{aligned}
				\|U_{1/2+it}G\|_{\BMO_{\partial,H}}
				&\le\|R(\Delta^{-it}G)\|_{\BMO_{\partial,\ell_2^n(H)}} = \|\Delta^{-it}G\|_{\BMO_{\Gamma,H}}\\
				&\lesssim e^{\pi|t|}\|G\|_{\BMO_{\Gamma,H}} \lesssim e^{\pi|t|}\|G\|_{L_\infty(\Om;H)}.
			\end{aligned}
		\end{equation*}
		
		Rescale to the unit strip by
		\[
		V_w=U_{w/2},\qquad0\le\Re w\le1,
		\]
		and let
		\begin{equation*}
			\theta=1-\frac2q.
		\end{equation*}
		If $w=iy$, we have
		\[
		\|V_{iy}\|_{L_2^0(\Omega_n;H)\to L_2^0(\Omega_n;H)}\le1.
		\]
		If $w=1+iy$, we also obtain
		\begin{equation}
			\|V_{1+iy}\|_{L_\infty^0(\Omega_n;H)\to\BMO_{\partial,H}}
			\lesssim e^{\pi|y|/2}.
			\label{eq:upper-boundary-growth}
		\end{equation}
		
		Now define
		\begin{equation*}
			\widetilde V_w=e^{(w-\theta)^2}V_w.
		\end{equation*}
		At $w=\theta$ the scalar factor is one. Note that
		\[
		\Re(iy-\theta)^2=\theta^2-y^2,
		\]
		and
		\[
		\Re(1+iy-\theta)^2=(1-\theta)^2-y^2.
		\]
		After multiplication with \eqref{eq:upper-boundary-growth}, the $y$-dependent exponent is
		$-y^2+(\pi/2)|y|$.  The elementary identity
		\begin{equation*}
			- y^2+b|y|
			=-\left(|y|-\frac{b}{2}\right)^2
			+\frac{b^2}{4}
			\le\frac{b^2}{4}
		\end{equation*}
		shows that the upper bound is uniform.  Thus $\widetilde V_w$ satisfies
		all hypotheses of Theorem~\ref{ext:Stein}. Since $\widetilde V_\theta=V_\theta$, we obtain the boundedness of
		\begin{equation*}
			V_\theta:
			[L_2^0(\Omega_n;H),L_\infty^0(\Omega_n;H)]_\theta
			\longrightarrow
			[L_2^0(\Omega_n;H),\BMO_{\partial,H}]_\theta.
		\end{equation*}
		
		It is clear that Theorem~\ref{ext:basicinterpolation}(iii)--(iv) yields
		\begin{equation*}
			[L_2^0(\Omega_n;H),L_\infty^0(\Omega_n;H)]_\theta=L_q^0(\Omega_n;H)
		\end{equation*}
		with equivalent norms, because
		\[
		\frac1q=\frac{1-\theta}{2}+\frac{\theta}{\infty}
		=\frac{1-\theta}{2}.
		\]
		By Lemma~\ref{lem:BMOtools}(iii), $	[L_2^0(\Omega_n;H),\BMO_{\partial,H}]_\theta=L_q^0(\Omega_n;H)$. Hence
		\[
		\|V_\theta G\|_{L_q^0(\Omega_n;H)}\lesssim_q\|G\|_{L_q^0(\Omega_n;H)}.
		\]
		Finally
		\[
		\frac{\theta}{2}=\frac12-\frac1q=a,
		\]
		so $V_\theta=U_a$, proving \eqref{eq:diagonal-estimate} for coordinatewise mean-zero $G$.
		
		For arbitrary $G$, let $G_0=G-\E G$.  The convention \eqref{eq:fractional-powers} implies $U_aG=U_aG_0$. Note that
		\[
		\|G_0\|_{L_q(\Om; H)}
		\le\|G\|_{L_q(\Om; H)}+\|\E G\|_{L_q(\Om; H)}
		\le2\|G\|_{L_q(\Om; H)}.
		\]
		Applying the mean-zero case to $G_0$ proves \eqref{eq:diagonal-estimate} for all $G$.
	\end{proof}
	
	Another auxiliary and remarkable result is the following inequality established by Ivanisvili and Volberg \cite[Formula~(6.5)]{IvanisviliVolberg}.
	
	\begin{lemma}
		\label{ext:IV}
		For every $2\le q<\infty$, every $n\ge1$, and scalar functions
		$f_1,\ldots,f_n$ on $\Om$,
		\begin{equation*}
			\left\|\sum_{j=1}^nD_j\Delta^{-1/2}f_j\right\|_{q}
			\lesssim_q \left\|\left(\sum_{j=1}^n|D_jf_j|^2\right)^{1/2}
			\right\|_{q}.
		\end{equation*}
	\end{lemma}

	In the following, fix $1<p<2$ and set
	\begin{equation*}
		q=\frac{p}{p-1}>2.
	\end{equation*}
	Then $1/p+1/q=1$. Define
	\begin{equation*}
		T_p=\nabla\Delta^{-1/p}:
		L_p(\Om)\longrightarrow L_p(\Om;\elln).
	\end{equation*}
	This operator annihilates constants. The following proposition concerns the boundedness of $T_p$.
	
	\begin{proposition}\label{prop:Tp}
		The norm of $T_p$ is bounded by a constant depending only on $p$,
		uniformly in $n$.
	\end{proposition}
	
	\begin{proof}
		Put $a=1/2-1/q$.  Since $1/q=1-1/p$,
		\begin{equation*}
			a=\frac12-\left(1-\frac1p\right)
			=\frac1p-\frac12,
			\qquad
			\frac12+a=\frac1p.
		\end{equation*}
		Let $G=(g_1,\ldots,g_n)\in L_q(\Om;\elln)$.  Applying Lemma~\ref{ext:IV} to
		$ f_j=\Delta^{-a}g_j$ 
		yields
		\begin{equation}\label{eq:adjoint-factorization-bound}
			\begin{aligned}
				\left\|\sum_{j=1}^nD_j\Delta^{-1/2}
				(\Delta^{-a}g_j)\right\|_q
				&\lesssim_q
				\left\|\bigl(D_j\Delta^{-a}g_j\bigr)_{j=1}^n
				\right\|_{L_q(\Om; \elln)}\\
				&\lesssim_q\|G\|_{L_q(\Om; \elln)}.
			\end{aligned}
		\end{equation}
		Here the last inequality is from Lemma~\ref{lem:diagonal}. Since
		\[
		\Delta^{-1/2}\Delta^{-a}
		=\Delta^{-(1/2+a)}=\Delta^{-1/p},
		\]
		we obtain for arbitrary $g_j$,
		\begin{equation}
			\sum_{j=1}^nD_j\Delta^{-1/2}\Delta^{-a}g_j
			=\Delta^{-1/p}\sum_{j=1}^nD_jg_j.
			\label{eq:power-factorization}
		\end{equation}
		
		Now we apply duality. For $h\in L^p(\Om)$, one has
		\begin{equation*}
			\begin{aligned}
				\langle T_ph,G\rangle
				&=\sum_{j=1}^n\E\left[
				D_j\Delta^{-1/p}h\,\overline{g_j}\right] =\sum_{j=1}^n\E\left[
				\Delta^{-1/p}h\,\overline{D_jg_j}\right]\\
				&=\sum_{j=1}^n\E\left[
				h\,\overline{\Delta^{-1/p}D_jg_j}\right] =\E\left[h\,\overline{
					\Delta^{-1/p}\sum_{j=1}^nD_jg_j}\right].
			\end{aligned}
		\end{equation*}
		Thus
		\begin{equation*}
			T_p^*G=\Delta^{-1/p}\sum_{j=1}^nD_jg_j.
		\end{equation*}
		Equations \eqref{eq:adjoint-factorization-bound}--\eqref{eq:power-factorization} give
		\[
		\|T_p^*G\|_q\lesssim_q \|G\|_{L_q(\Om; \elln)}.
		\]
		Therefore
		\[
		\|T_p\|_{L^p\to L^p(\Om; \elln)}
		=\|T_p^*\|_{L_q(\Om; \elln)\to L_q(\Om)}
		\lesssim_q 1.
		\]
		All constants depend only on $q=p/(p-1)$ but not on $n$.
	\end{proof}
	
	Finally, we come to the proof of Theorem \ref{thm:endpoint}.
	\begin{proof}[Proof of Theorem~\ref{thm:endpoint}]
		For arbitrary $f$, put $h=\Delta^{1/p}f$.  Formula \eqref{eq:fractional-powers} shows
		$h\in L_p^0$ and
		\[
		\Delta^{-1/p}h
		=\sum_{\varnothing\ne A\subseteq[n]}\widehat f(A)w_A
		=f-\widehat f(\varnothing)=f-\E f.
		\]
		By Proposition~\ref{prop:Tp}, we have
		\begin{align*}
			\|\nabla f\|_{L_p(\Om; \elln)}
			&=\|\nabla(f-\E f)\|_{L_p(\Om; \elln)} =\|\nabla\Delta^{-1/p}h\|_{L_p(\Om; \elln)}\\
			&=\|T_ph\|_{L_p(\Om; \elln)}\lesssim_q\|h\|_{p}=\|\Delta^{1/p}f\|_{p}.
		\end{align*}
		This finishes the proof.
	\end{proof}
	
	\
	
	{\bf AI Statement.} The authors acknowledge the use of AI tools, including ChatGPT, for language
	polishing, LaTeX editing, and exploratory mathematical discussions during the development and
	preparation of this manuscript. Some ideas used in the proof-development process arose during
	interactions with GPT-5.6 Sol. Suggestions from these AI interactions were subsequently examined, reformulated, incorporated into the manuscript, and independently verified by
	the authors. The authors take full responsibility for all mathematical content in the final manuscript.

	{\bf Acknowledgments.} Hao Zhang would like to thank Professor Alexander Volberg and Professor Quanhua Xu for nice discussions on this problem on May 2024 in Besançon. The authors thank Professor Alexandros Eskenazis for helpful communications, during which he pointed out the existence of Naor's original proof and an application of Theorem 1.1. The authors also thank Professor Quanhua Xu for helpful suggestions.

\end{document}